\documentclass[final,leqno]{siamltex}
\usepackage{amsmath}
\usepackage{graphicx}
\usepackage[notcite,notref]{showkeys}
\usepackage{mathrsfs}
\usepackage{graphics}
\usepackage{float}
\usepackage{amsfonts,amssymb}
\usepackage{dsfont}
\usepackage{pifont}
\usepackage{wrapfig} 
\usepackage{hyperref}
\usepackage{multirow}
\usepackage{color}
\numberwithin{equation}{section}
\usepackage{threeparttable}
\def\E{{\mathcal{E}}}
\def\T{{\mathcal{T}}}

\def\cal#1{{\mathcal #1}}
\def\pT{{\partial T}}

\def\bw{{\mathbf{w}}}

\def\bf{{\mathbf{f}}}
\def\bu{{\mathbf{u}}}
\def\bg{{\mathbf{g}}}
\def\bv{{\mathbf{v}}}
\def\bn{{\mathbf{n}}}
\def\be{{\mathbf{e}}}

\def\beta{{\boldsymbol{\eta}}}
\def\bvarphi{{\boldsymbol{\varphi}}}

\def\ljump{{[\![}}
\def\rjump{{]\!]}}
\def\bsigma{{\boldsymbol{\sigma}}}
   
\def\3bar{{|\hspace{-.02in}|\hspace{-.02in}|}}
\def\E{{\mathcal{E}}}
\def\T{{\mathcal{T}}}

\def\bpsi{\boldsymbol{\psi}}

\def\pT{{\partial T}}

\def\bw{{\mathbf{w}}}
\def\bu{{\mathbf{u}}}
\def\bg{{\mathbf{g}}}
\def\bv{{\mathbf{v}}}
\def\bn{{\mathbf{n}}}
\def\bQ{{\mathbf{Q}}}

\def\be{{\mathbf{e}}}
\def\bw{{\mathbf{w}}}

\def\bf{{\mathbf{f}}}
\def\bphi{{\boldsymbol{\phi}}}
\def\bsigma{{\boldsymbol{\sigma}}}
\def\bepsilon{{\boldsymbol{\epsilon}}}

\def\bphi{{\boldsymbol{\phi}}}
\def\ljump{{[\![}}
\def\rjump{{]\!]}}

\def\bsigma{{\boldsymbol{\sigma}}}

\newtheorem{remark}{Remark}[section]
\newtheorem{algorithm}{Auto-Stablized Weak Galerkin Algorithm}[section]
 
\def\ad#1{\begin{aligned}#1\end{aligned}}  \def\b#1{{\mathbf{#1}}}
\def\a#1{\begin{align*}#1\end{align*}} \def\an#1{\begin{align}#1\end{align}} 
 \def\t#1{\hbox{\rm{#1}}}
\def\p#1{\begin{pmatrix}#1\end{pmatrix}}  \numberwithin{equation}{section}

\title {An Auto-Stabilized Weak Galerkin Method  for Elasticity Interface Problems on Nonconvex Meshes }

\begin{document}

\author{
Chunmei Wang \thanks{Department of Mathematics, University of Florida, Gainesville, FL 32611, USA (chunmei.wang@ufl.edu). The research of Chunmei Wang was partially supported by National Science Foundation Award DMS-2136380.}
 \and  
 Shangyou Zhang\thanks{Department of Mathematical Sciences,  University of Delaware, Newark, DE 19716, USA (szhang@udel.edu).  }}

\maketitle

\begin{abstract}
This paper introduces an auto-stabilized weak Galerkin (WG) finite element method for elasticity interface problems on general polygonal and polyhedral meshes, without requiring convexity constraints. The method utilizes bubble functions as key analytical tools, eliminating the need for stabilizers typically used in traditional WG methods and leading to a more streamlined formulation. The proposed method is symmetric, positive definite, and easy to implement. Optimal-order error estimates are derived for the WG approximations in the discrete $H^1$-norm, assuming the exact solution has sufficient smoothness. Numerical experiments validate the accuracy and efficiency of the auto-stabilized WG method.

\end{abstract}



\pagestyle{myheadings}

\section{Introduction}
This paper focuses on the development of weak Galerkin (WG) finite element methods for addressing elasticity interface problems. Specifically, consider a bounded domain  $\Omega\subset \mathbb R^d$ ($d = 2, 3$) with a piecewise smooth Lipschitz boundary  $\partial \Omega$. Let $N$ and $M$ be positive integers. The domain $\Omega$ is partitioned into a collection of subdomains $\{\Omega_i\}_{i=1}^N$, each with a piecewise smooth Lipschitz boundary $\partial \Omega_i$ for $i=1, \cdots, N$.  The interface between these subdomains is denoted by  $\Gamma=\bigcup_{i=1}^N \partial \Omega_i\setminus \partial \Omega$ and can be expressed as
$$
\Gamma =\bigcup_{m=1}^{M} \Gamma_m,
$$
where for each  $m=1, \cdots, M$,  there exist 
$i, j \in \{1, \cdots, N\}$ such that $\Gamma_m=\partial \Omega_i \cap\partial \Omega_j$. The elasticity interface problem is formulated as follows: find the displacement  $\bu$ such that
\begin{equation}\label{model}
\begin{split}
\bsigma(\bu_i)&=2\mu_i\bepsilon(\bu_i)+\lambda_i\nabla \cdot \bu_i I, \qquad \text{in}\ \Omega_i, i=1, \cdots, N,\\
-\nabla \cdot \bsigma(\bu_i) &=\bf_i, \qquad\qquad\qquad\qquad\qquad \text{in}\ \Omega_i,i=1, \cdots, N,\\
\ljump \bu\rjump_{\Gamma_m}&=\bphi_m, \qquad\qquad\qquad\qquad\qquad \text{on}\ \Gamma_m, m=1,\cdots, M,\\
\ljump \bsigma(\bu)\bn\rjump_{\Gamma_m}&=\bpsi_m, \qquad\qquad\qquad\qquad\qquad \text{on}\ \Gamma_m,m=1,\cdots, M,\\
\bu_i&=\bg_i,\quad\qquad \qquad\qquad\qquad\qquad \text{on}\ \partial\Omega_i\cap \partial\Omega, i=1,\cdots, N,
\end{split}
\end{equation}
where $\bu_i=\bu|_{\Omega_i}$, $\bf_i=\bf|_{\Omega_i}$, $\bg_i=\bg|_{\Omega_i}$, $\mu_i=\mu|_{\Omega_i}$, $\lambda_i=\lambda|_{\Omega_i}$,   $[\![\bsigma(\bu)  \bn]\!]_{\Gamma_m}=\bsigma(\bu_i)  \bn_i+ \bsigma(\bu_j)  \bn_j$, and $[\![\bu]\!]_{\Gamma_m}=\bu_i|_{  \Gamma_m}-\bu_j|_{\Gamma_m}$. Here, $\bn_i$ and $\bn_j$ denote the outward unit normal directions to  $\partial \Omega_i\cap \Gamma_m$ and $\partial \Omega_j\cap \Gamma_m$, respectively.  
Throughout this paper, boldface letters represent vector-valued functions and their associated function spaces. The strain tensor is given by $\bepsilon=\frac{1}{2} (\nabla \bu+\nabla\bu^T)$, $\bf$ represents the body force, and $\lambda$ and $\mu$ are the positive Lam$\acute{e}$ parameters. These parameters are related to the Young’s modulus  $E$ and the Poisson's ratio $\nu$  through the following expressions:
$$
\lambda=\frac{E\nu}{(1-2\nu)(1+\nu)}, \qquad \mu=\frac{E}{2(1+\nu)}.
$$
The Lam$\acute{e}$ parameters $\lambda$ and $\mu$ are assumed to be piecewise smooth   with respect to the partition $\Omega=\bigcup_{i=1}^N \Omega_i$ and  $\frac{\lambda}{\mu}=\frac{2\nu}{1-2\nu}$ is assumed to be bounded.

The weak formulation of the elasticity interface model problem \eqref{model} is as follows: Find $\bu$ such that $\bu_i=\bg_i$ on $\partial\Omega_i\cap\partial \Omega$ ($i=1, \cdots, N$),  and
\begin{equation}\label{weakform}
(2\mu\epsilon(\bu), \epsilon(\bv))+(\lambda \nabla\cdot\bu, \nabla\cdot\bv)=(\bf, \bv)+\sum_{m=1}^M\langle\bpsi_m,\bv\rangle_{\Gamma_m}, \quad\forall \bv\in [H_0^1(\Omega)]^d.
\end{equation}
 
Elasticity interface problems are central to continuum mechanics, where elasticity theory and the associated partial differential equations (PDEs) are used to describe a wide range of material behaviors. These problems are especially important in scenarios involving voids, pores, inclusions, dislocations, cracks, or composite structures within materials \cite{w1, w2, w3, w4}, where an interface-based description becomes indispensable. Such problems find significant applications in tissue engineering, biomedical science, and biophysics \cite{w5, w6, w7}. In many cases, interfaces are dynamic, as in fluid–structure interfacial boundaries \cite{w8}, and material properties often exhibit discontinuities across these interfaces.
For elastic bodies composed of heterogeneous materials with distinct physical properties, the governing equations apply individually to each subdomain. The solution must additionally satisfy the displacement and traction jump conditions at the interfaces between different materials, alongside the standard boundary conditions.
In linear elasticity theory, the stress–strain relationship is defined by constitutive equations. For isotropic and homogeneous materials, these equations can be expressed using any two of the following parameters: bulk modulus, Young’s modulus, Lamé’s first parameter, shear modulus, Poisson’s ratio, or P-wave modulus \cite{w9}. When these parameters vary spatially, the constitutive equations describe the elastic properties of isotropic inhomogeneous media. In seismic wave modeling, such inhomogeneity is often captured by allowing Lamé’s parameters to depend on position \cite{w10}.

A variety of numerical methods have been developed to solve elasticity interface problems. The immersed interface method (IIM) was introduced for elasticity problems in isotropic homogeneous media \cite{w15, w16, w17}, and a second-order sharp numerical method was designed for linear elasticity equations \cite{w18}. Finite element methods, including the partition of unity method (PUM), generalized finite element method (GFEM), and extended finite element method (XFEM), were developed to handle the non-smooth solution behavior across interfaces by incorporating enrichment functions into the approximations \cite{w2, w3, w4}. Discontinuous Galerkin methods have been employed to model strong and weak discontinuities through weak enforcement of continuity conditions \cite{w19, w20, w21}.
The immersed finite element method (IFEM) was proposed to address elasticity problems with inhomogeneous jump conditions \cite{w22, w23, w24}, and a sharp-edged interface approach was developed for a specific class of elasticity interface problems \cite{w25}. The bilinear IFEM was further enhanced into a locking-free version \cite{w26, w27}. Additionally, the immersed meshfree Galerkin method was proposed for composite solids \cite{w28}, while a Nitsche-type method was developed to tackle elasticity interface problems \cite{w29}.  

  The  WG  finite element method has transformed the numerical framework for solving  PDEs. By employing distributions and piecewise polynomials, WG extends beyond the limitations of traditional finite element methods. Unlike classical approaches, WG relaxes strict regularity requirements for function approximations, instead utilizing well-designed stabilizers to ensure stability and accuracy. Recent research has thoroughly demonstrated WG's versatility in addressing a broad spectrum of model PDEs, establishing it as a robust and dependable tool in computational science \cite{wg1, wg2, wg3, wg4, wg5, wg6, wg7, wg8, wg9, wg10, wg11, wg12, wg13, wg14, wg15, wg16, wg17, wg18, wg19, wg20, wg21, itera, wy3655}. Its strength lies in leveraging weak derivatives and weak continuities to design numerical schemes grounded in the weak forms of underlying PDEs.
A significant development within the WG framework is the Primal-Dual Weak Galerkin (PDWG) method \cite{pdwg1, pdwg2, pdwg3, pdwg4, pdwg5, pdwg6, pdwg7, pdwg8, pdwg9, pdwg10, pdwg11, pdwg12, pdwg13, pdwg14, pdwg15}. PDWG formulates numerical solutions as constrained minimizations of functionals, where the constraints are derived from the weak formulation of PDEs using weak derivatives. This approach leads to an Euler-Lagrange system that incorporates both the primal variable and the dual variable (Lagrange multiplier), resulting in a symmetric and efficient numerical scheme.

This paper introduces an auto-stabilized weak Galerkin  finite element method that eliminates the reliance on stabilizers. The proposed approach is applicable to polytopal meshes without convexity constraints. The key innovation enabling this advancement is the use of bubble functions. As a trade-off, this method requires higher-degree polynomials for computing the discrete weak strain tensor and discrete weak divergence.
Despite this requirement, our method preserves the size and global sparsity of the stiffness matrix, significantly reducing programming complexity compared to traditional stabilizer-dependent WG methods. Theoretical analysis demonstrates that the WG approximations achieve optimal error estimates in the discrete 
$H^1$
  norm. By offering a stabilizer-free WG method that maintains high performance while reducing computational complexity, this work makes a notable contribution to the development of finite element methods on non-convex polytopal meshes.

The structure of this paper is as follows: Section 2 provides a concise overview of weak differential operators and their discrete counterparts. Section 3 introduces the auto-stabilized weak Galerkin scheme, detailing its construction and features. Section 4 establishes the theoretical groundwork by proving the existence and uniqueness of the solution for the algorithm developed in Section 3. In Section 5, we derive the error equation associated with the method, which leads naturally to Section 6, where error estimates in the energy norm are rigorously analyzed. Finally, Section 7 presents numerical experiments that corroborate the theoretical results and demonstrate the effectiveness of the proposed approach.

Throughout this paper, we adopt standard notations. Let $D$ represent any open bounded domain in $\mathbb{R}^d$ with a Lipschitz continuous boundary.  The inner product, semi-norm and norm in the Sobolev space $H^s(D)$ for any integer $s\geq 0$ are denoted by $(\cdot,\cdot)_{s,D}$, $|\cdot|_{s,D}$ and $\|\cdot\|_{s,D}$, respectively. For simplicity, when the domain $D$ is chosen as $D=\Omega$, the subscript $D$ is  omitted from the notations. In the case where $s=0$, the notations $(\cdot,\cdot)_{0,D}$, $|\cdot|_{0,D}$ and $\|\cdot\|_{0,D}$ are further simplified as $(\cdot,\cdot)_D$, $|\cdot|_D$ and $\|\cdot\|_D$, respectively.

\section{Discrete Weak Strain Tensor and Discrete Weak Divergence}\label{Section:Hessian}
In this section, we provide a brief overview of the weak strain tensor and weak divergence, along with their discrete counterparts as introduced in \cite{wg10, wg18}.

Let $T$ be a polytopal element with boundary $\partial T$. A weak function on $T$  is defined as   $\bv=\{\bv_0, \bv_b\}$, where $\bv_0\in [L^2(T)]^d$ represents the value of $\bv$ in the interior of $T$, and $\bv_b\in [L^{2}(\partial T)]^d$ represents   the value of $\bv$   on the boundary of $T$. In general, $\bv_b$ is treated as independent of the trace of $\bv_0$. A special case occurs when $\bv_b= \bv_0|_{\partial T}$, where the function $
\bv=\{\bv_0, \bv_b\}$ is fully determined by $\bv_0$ and can be denoted simply as $\bv=\bv_0$.
 
 The space of all weak functions on $T$,  denote by $W(T)$, is defined as:
$$ W(T)=\{v=\{\bv_0,\bv_b\}: \bv_0\in [L^2(T)]^d, \bv_b\in [L^{2}(\partial
 T)]^d\}. $$
 
 The weak gradient, denoted by $\nabla_{w}$, is a linear
 operator mapping $W(T)$ to the dual space of $[H^{1}(T)]^{d\times d}$. For any
 $\bv\in W(T)$, the weak gradient  $\nabla_w \bv$ is defined as a bounded linear functional on $[H^{1}(T)]^{d\times d}$
such that
 \begin{equation*}\label{2.3}
  (\nabla _{w}\bv, \bvarphi)_T=-(\bv_0,\nabla\cdot \bvarphi)_T+
  \langle \bv_b, \bvarphi\cdot \bn \rangle_{\partial T},\quad \forall \bvarphi\in [H^{1}(T)]^{d\times d},
  \end{equation*}
 where $ \bn$ denotes the outward unit normal vector to $\partial T$.
 
 For any non-negative integer $r$, let $P_r(T)$  denote the space of polynomials on $T$ with total degree at most 
 $r$. A discrete weak
 gradient on $T$, denoted by $\nabla_{w, r_1, T}$,  is defined as a linear operator mapping 
  $W(T)$ to $[P_{r_1}(T)]^{d\times d}$. For any $\bv\in W(T)$, the discrete weak gradient
 $\nabla_{w, r_1, T}\bv$ is the unique polynomial matrix in $[P_{r_1}(T)]^{d\times d}$ that satisfies:
 \begin{equation*}\label{2.4}
  (\nabla_{w,r_1,T} \bv, \bvarphi)_T=-(\bv_0,\nabla\cdot \bvarphi)_T+
  \langle \bv_b,  \bvarphi\cdot \bn \rangle_{\partial T},\quad \forall \bvarphi \in [P_{r_1}(T)]^{d\times d}.
  \end{equation*}
  
The discrete weak strain tensor is defined as:
$$
\epsilon_{w,r_1,T}(\bu)=\frac{1}{2}(\nabla_{w,r_1,T}\bu+\nabla_{w,r_1,T}\bu^T).
$$
For any $\bv\in W(T)$, the discrete weak strain tensor, denoted by
 $\epsilon_{w, r_1, T}(\bv)$,  is the unique polynomial matrix in $[P_{r_1}(T)]^{d\times d}$ that satisfies:
 \begin{equation}\label{2.5}
  (\epsilon_{w,r_1,T} (\bv), \bvarphi)_T=-(\bv_0,\nabla  \cdot \frac{1}{2}(\bvarphi+\bvarphi^T) )_T+
  \langle \bv_b,  \frac{1}{2}(\bvarphi+\bvarphi^T)\cdot \bn \rangle_{\partial T}, 
  \end{equation}
for all $\bvarphi \in [P_{r_1}(T)]^{d\times d}$.
  
 If $\bv_0\in [H^1(T)]^d$ is sufficiently smooth, applying the standard integration by parts to the first term on the right-hand side of \eqref{2.5} leads to:
 \begin{equation}\label{2.5new}
 (\epsilon_{w,r_1,T} (\bv), \bvarphi)_T=(\epsilon( \bv_0), \bvarphi)_T+
  \langle \bv_b-\bv_0, \frac{1}{2}( \bvarphi+\bvarphi^T)\cdot \bn \rangle_{\partial T}.  
  \end{equation} 
 for all $\bvarphi \in [P_{r_1}(T)]^{d\times d}$.

 The weak divergence of $\bv\in W(T)$, denoted by $\nabla_w\cdot \bv$, is a bounded linear functional in the Sobolev space $H^1(T)$. Its action on any $\phi\in H^1(T)$ is given by
\begin{equation*}\label{div}
    (\nabla_w\cdot \bv, \phi)_T=-(\bv_0, \nabla\phi)_T+\langle \bv_b\cdot\bn, \phi\rangle_{\partial T}.
\end{equation*}

The discrete weak divergence of $\bv\in W(T)$, denoted by $\nabla_{w, r_2, T}\cdot \bv$, is the unique polynomial in $P_{r_2}(T)$ that satisfies:
\begin{equation}\label{disdiv}
    (\nabla_{w, r_2, T}\cdot \bv, \phi)_T=-(\bv_0, \nabla\phi)_T+\langle \bv_b\cdot\bn, \phi\rangle_{\partial T},
\end{equation}
for all $\phi\in P_{r_2}(T)$. 

 For a smooth $\bv_0\in
 [H^1(T)]^d$,  applying the standard integration by parts to the first term on the right-hand side  of (\ref{disdiv})  gives
\begin{equation}\label{disdivnew}
    (\nabla_{w, r_2, T}\cdot \bv, \phi)_T= (\nabla\cdot \bv_0,  \phi)_T+\langle (\bv_b-\bv_0)\cdot\bn, \phi\rangle_{\partial T},
\end{equation}
for all $\phi\in P_{r_2}(T)$.

\section{Auto-Stabilized Weak Galerkin Algorithm}\label{Section:WGFEM}
 
 Let ${\cal T}_h$ be a finite element partition of the domain
 $\Omega\subset \mathbb R^d$ into polytopal elements, where   ${\cal
 T}_h$ is assumed to satisfy the shape-regularity condition as defined in \cite{wy3655}. Denote by ${\mathcal E}_h$ the set of all edges (or faces) of 
 ${\cal T}_h$, and let ${\mathcal E}_h^0={\mathcal E}_h \setminus
 \partial\Omega$ be the set of interior edges or faces. For any element  $T\in {\cal T}_h$, its diameter is denoted by  $h_T$, and the mesh size of the partition is  $h=\max_{T\in {\cal
 T}_h}h_T$.

For each element  $T\in {\cal T}_h$, let $RM(T)$ represent the space of rigid motions on $T$,  defined as
$$
RM(T)=\{\textbf{a}+\eta \textbf{x}: \textbf{a}\in \mathbb R^d, \eta\in so(d)\},
$$
where $\textbf{x}$ is the position vector on $T$ and $so(d)$ denotes the space of skew-symmetric $d\times d$ matrices. The trace of the rigid motion on each edge $e\subset T$ forms a finite-dimensional space denoted by  $P_{RM} (e)$,  i.e.,
$$
P_{RM} (e)=\{ \bv\in [L^2(e)]^d: \bv=\tilde{\bv}|_e \ \text{for some}\ \tilde{\bv}\in RM(T), e\subset \partial T\}.
$$

 For each element $T\in\T_h$, the local weak finite element space is defined as:
 \begin{equation}\label{Vk}
 V(k, T)=\{\{\bv_0,\bv_b\}: \bv_0\in [P_k(T)]^d,\bv_b\in
S_k(e),  e\subset \partial T\}.   
 \end{equation}
Here, $S_k(e)= [P_{k-1}(e)]^d+P_{RM} (e)$. Since $P_{RM} (e)\subset P_1(e)$,   the boundary component $S_k(e)$ simplifies to $[P_{k-1} (e)]^d$ for $k>1$, and to $P_{RM} (e)$ when $k=1$.

By assembling $V(k, T)$ across all elements  $T\in {\cal T}_h$ and imposing continuity on the interior interfaces $\E_h^0$,
the global weak finite element space is defined as:
\begin{equation}\label{Vh}
 V_h=\big\{\{\bv_0,\bv_b\}:\ \{\bv_0,\bv_b\}|_T\in V(k,  T),
 \forall T\in {\cal T}_h \big\}.
 \end{equation} 
The subspace of $V_h$ with vanishing boundary values on $\partial\Omega$ is given by:
\begin{equation}\label{Vh0}
V_h^0=\{\{\bv_0,\bv_b\}\in V_h: \bv_b=0 \ \text{on}\ \partial\Omega\}.
\end{equation}

To simplify notation, the discrete weak strain tensor $\epsilon_{w, r_1, T}\bv$ and the discrete weak divergence $\nabla_{w, r_2, T} \cdot\bv$  are denoted by  
  $\epsilon_{w}\bv$ and $\nabla_{w} \cdot\bv$, respectively. These quantities are computed locally on each element $T$ using the definitions in 
\eqref{2.5} and \eqref{disdiv}: 
\an{\label{w-d} \ad{ 
(\epsilon_{w} \bv)|_T & = \epsilon_{w, r_1, T}(\bv |_T), \qquad \forall T\in \T_h, \\
(\nabla_{w}\cdot \bv)|_T&= \nabla_{w, r_2, T}\cdot(\bv |_T), \qquad \forall T\in \T_h. } }

For any $\bu, \bv\in V_h$, the following bilinear form is introduced:
\begin{equation}\label{EQ:local-bterm}
  a(\bu, \bv)=  \sum_{T\in {\cal T}_h}a_T(\bu, \bv),   
\end{equation} 
where
\begin{equation*} 
a_T(\bu, \bv)= (2\mu \epsilon_w(\bu), \epsilon_w(\bv))_T+(\lambda \nabla_w\cdot\bu, \nabla_w\cdot\bv)_T. 
\end{equation*}

Let $\bQ_b$ denote the $L^2$ projection onto the space $P_k(e)$. The auto-stabilized weak Galerkin scheme for the elasticity interface problem \eqref{model}, based on the variational formulation \eqref{weakform}, is described below: 
\begin{algorithm}\label{PDWG}
Find $\bu_h \in  V_{h}$, such that  
$\bu_b=\bQ_b\bg_i$ on $\partial\Omega_i\cap\partial\Omega$ for $i=1, \cdots, N$ and $\bu_b^L-\bu_b^R=\bQ_b\bphi_m$ on $\Gamma_m$ for $m=1,\cdots, M$, satisfying:
\begin{equation}\label{WG}
 a(\bu, \bv)=(\bf, \bv_0)+\sum_{m=1}^M\langle \bpsi_m, \bv_b\rangle_{\Gamma_m},  
\end{equation}
for any $\bv\in V_h^0$. Here, $\bu_b^L$ and $\bu_b^R$ represent the restrictions of $\bu_b$ to $\Omega_i$ and   $\Omega_j$, respectively, where $\Gamma=\partial\Omega_i\cap\partial\Omega_j$ for some $i$ and $j$. 
\end{algorithm}

\section{Solution Existence and Uniqueness} 
 
We begin by recalling key trace inequalities. For a shape-regular finite element partition ${\cal T}_h$ of the domain $\Omega$,  the following trace inequality holds for any element   $T\in {\cal T}_h$ and function $\phi\in H^1(T)$ \cite{wy3655}: 
\begin{equation}\label{tracein}
 \|\phi\|^2_{\partial T} \leq C(h_T^{-1}\|\phi\|_T^2+h_T \|\nabla \phi\|_T^2).
\end{equation}
For polynomial functions  $\phi$, a simplified trace inequality is used  \cite{wy3655}: 
\begin{equation}\label{trace}
\|\phi\|^2_{\partial T} \leq Ch_T^{-1}\|\phi\|_T^2.
\end{equation}

Next, we define two essential norms for error analysis. For any $\bv=\{\bv_0, \bv_b\}\in V_h$, the discrete energy norm is defined as: \begin{equation}\label{3norm}
\3bar \bv\3bar=\Big( \sum_{T\in {\cal T}_h} (2\mu\epsilon_w (\bv), \epsilon_w(\bv))_T+(\lambda \nabla_w \cdot \bv, \nabla_w \cdot \bv)_T\Big)^{\frac{1}{2}},
\end{equation}
and the discrete $H^1$ semi-norm  is given by:
\begin{equation}\label{disnorm}
\|\bv\|_{1, h}=\Big(\sum_{T\in {\cal T}_h} (2\mu\epsilon  (\bv_0), \epsilon (\bv_0))_T+(\lambda \nabla  \cdot \bv_0, \nabla  \cdot \bv_0)_T+h_T^{-1}\|\bv_0-\bv_b\|_{\partial T}^2\Big)^{\frac{1}{2}}.
\end{equation}

The auto-stabilized WG method is applicable to polytopal meshes without convexity constraints, with bubble functions serving as a key analytical tool. For completeness, we present Lemmas 4.2, 4.3, 4.5, 4.6, and 4.7 in this paper. However, the proofs of these lemmas are omitted to avoid redundancy, as they can be found in detail in \cite{wangelas}.

\begin{definition}  (Element-based bubble function)
   Let  $T\in {\cal T}_h$ be a polytopal element with $N$ edges/faces denoted by $e_1, \cdots, e_N$. Note that $T$ may be non-convex.  For each edge/face $e_i$, define a linear function  $l_i(x)$ such that $l_i(x)=0$ on $e_i$:
  $$l_i(x)=\frac{1}{h_T}\overrightarrow{AX}\cdot \bn_i, $$  where  $A$ is a point on  $e_i$,  $X=(x_1, \cdots, x_{d-1})$ is an arbitrary point on   $e_i$, $\bn_i$ is the normal vector to   $e_i$, and $h_T$ is the diameter of element $T$. 

  The bubble function for $T$ is defined as 
   $$
   \Phi_B =l^2_1(x)l^2_2(x)\cdots l^2_N(x) \in P_{2N}(T).
  $$ 
  It is straightforward to verify that  $\Phi_B=0$ on $\partial T$.    The function 
   $\Phi_B$  can be scaled such that $\Phi_B(M)=1$ where   $M$  is the barycenter of $T$.  Furthermore,  there exists a sub-domain $\hat{T}\subset T$ such that $\Phi_B\geq \rho_0$ for some constant $\rho_0>0$. In this case, we choose  $r_1=2N+k-1$ and $r_2=2N+k-1$.  
\end{definition}

\begin{lemma}\cite{wang1, wang2, wangelas}\label{norm1}
 For $\bv=\{\bv_0, \bv_b\}\in V_h$, there exists a constant $C$ such that
 $$
 \|\epsilon(\bv_0)\|_T\leq C\|\epsilon_w (\bv)\|_T.
 $$
\end{lemma}

\begin{lemma}\cite{wangelas}\label{norm2}
 For $\bv=\{\bv_0, \bv_b\}\in V_h$, there exists a constant $C$ such that
 $$
 \|\nabla\cdot\bv_0\|_T\leq C\|\nabla_w\cdot \bv\|_T.
 $$
\end{lemma}

\begin{remark}
If the polytopal element $T$ is convex, 
   the bubble function  in Lemma \ref{norm1}  simplifies to
 $$
 \Phi_B =l_1(x)l_2(x)\cdots l_N(x).
 $$ 
It can be shown that there exists a subdomain $\hat{T}\subset T$ such that
 $ \Phi_B\geq\rho_0$  for some constant $\rho_0>0$  and $\Phi_B=0$ on the boundary $\partial T$.  Using this simplified construction,   Lemmas \ref{norm1}-\ref{norm2}   can be established in the same way. In this case, we choose  $r_1=N+k-1$ and $r_2=N+k-1$.  
\end{remark}

\begin{definition} (Edge/face-based bubble function)
    Recall that $T$ is a $d$-dimensional polytopal element and  $e_i$ is a $(d-1)$-dimensional edge/face  of $T$.  Define the edge/face-based bubble function as   $$\varphi_{e_i}= \Pi_{k=1, \cdots, N, k\neq i}l_k^2(x).$$ This bubble function satisfies the following properties: (1) $\varphi_{e_i}=0$ on  $e_k$ for $k \neq i$, and (2) there exists a subdomain $\widehat{e_i}\subset e_i$ such that $\varphi_{e_i}\geq \rho_1$ for some constant $\rho_1>0$.  
\end{definition}

\begin{lemma}\cite{wangelas}\label{phi}
     For $\bv=\{\bv_0, \bv_b\}\in V_h$, define  $\bvarphi=(\bv_b-\bv_0) \bn^T\varphi_{e_i}$, where $\bn$ is the unit outward normal to the edge/face  $e_i$. Then, the following inequality holds:
\begin{equation}
  \|\bvarphi\|_T ^2 \leq Ch_T \int_{e_i}((\bv_b-\bv_0)\bn^T)^2ds.
\end{equation}
\end{lemma}
  
\begin{lemma}\cite{wangelas}\label{phi2}
     For $\bv=\{\bv_0, \bv_b\}\in V_h$, define $\phi=(\bv_b-\bv_0) \cdot\bn \varphi_{e_i}$, where $\bn$ is the unit outward normal direction to the edge/face  $e_i$. Then, the following inequality holds: 
\begin{equation}
  \|\phi\|_T ^2 \leq Ch_T \int_{e_i}(\bv_b-\bv_0)^2ds.
\end{equation}
\end{lemma}

\begin{lemma}\cite{wangelas}\label{normeqva}   There exists  positive constants $C_1$ and $C_2$ such that,  for any $\bv=\{\bv_0, \bv_b\} \in V_h$,  
 \begin{equation}\label{normeq}
 C_1\|\bv\|_{1, h}\leq \3bar \bv\3bar  \leq C_2\|\bv\|_{1, h}.
\end{equation}
\end{lemma}

  \begin{remark}
   If the polytopal element $T$ is convex, 
  the edge/face-based bubble function in Lemmas \ref{phi}-\ref{normeqva}  simplifies to
$$\varphi_{e_i}= \Pi_{k=1, \cdots, N, k\neq i}l_k(x).$$
It can be shown that (1)  $\varphi_{e_i}=0$ on   $e_k$ for $k \neq i$, and (2) there exists a subdomain $\widehat{e_i}\subset e_i$ such that $\varphi_{e_i}\geq \rho_1$ for some constant $\rho_1>0$.  Using this simplified construction,  Lemmas \ref{phi}-\ref{normeqva}  can be derived in the same manner.
\end{remark}

\begin{remark}
For any  $d$-dimensional polytopal element $T$,  
 there exists a hyperplane $H\subset R^d$  such that a finite number $l$ of distinct $(d-1)$-dimensional edges/faces containing $e_{i}$ are contained  in $H$.  Under this setting, Lemmas \ref{phi}-\ref{normeqva} can be proved with additional techniques. For more details, see \cite{wang1, wang2, wangelas}, which provide a generalization applicable to Lemmas  \ref{phi}-\ref{normeqva}.
 \end{remark}

\begin{lemma}\cite{wg10, wg18} (Second Korn's Inequality)
Let $\Omega$  be a connected, open, bounded domain with a Lipschitz continuous boundary. Assume $\Gamma_1\subset \partial\Omega$ is a nontrivial portion of $\partial\Omega$  with dimension $d-1$. For any fixed real number $1\leq p<\infty$, there exists a constant $C$ such that
\begin{equation}\label{korn}
 \|\bv\|_1\leq C(\|\epsilon(\bv)\|_0+\|\bv\|_{L^p(\Gamma_1)}), 
\end{equation}
for any $\bv\in [H^1(\Omega)]^d$.
\end{lemma}

\begin{theorem}
The auto-stabilized weak Galerkin finite element scheme \eqref{PDWG} has a unique solution.
\end{theorem}
\begin{proof}
Since the number of equations matches the number of unknowns in \eqref{WG}, it suffices to prove uniqueness.  Let $\bu_h^{(1)}=\{\bu_0^{(1)}, \bu_b^{(1)}\}$ and $\bu_h^{(2)}=\{\bu_0^{(2)}, \bu_b^{(2)}\}\in V_h$ be two solutions of \eqref{WG}. Then, $\bu_b^{(j)} (j=1, 2)$ satisfies $\bu_b^{(j)}=\bQ_b\bg_i$ on $\partial\Omega_i\cap\partial\Omega$ for $i=1, \cdots, N$, $\bu_b^{(j), L}-\bu_b^{(j), R}=\bQ_b\bphi_m$ on $\Gamma_m$ for $m=1,\cdots, M$, and
\begin{equation}\label{tt1}
 a(\bu^{(j)}, \bv)=(\bf, \bv_0)+\sum_{m=1}^M\langle \bpsi_m, \bv_b \rangle_{\Gamma_m},  \qquad \forall \bv=\{\bv_0, \bv_b\}\in V_h^0, j=1, 2.
\end{equation} 
The difference   $\bw=\bu_h^{(1)}-\bu_h^{(2)} \in V_h^0$ satisfies 
\begin{equation}\label{tt2}
 a(\bw, \bv)=0, \qquad \forall \bv=\{\bv_0, \bv_b\} \in V_h^0.
\end{equation} 
Choosing  $\bv=\bw$ in \eqref{tt2} and using the norm equivalence \eqref{normeq}, we deduce that  $\|\bw\|_{1, h}=0$. This implies:
\begin{eqnarray}\label{c1}
\epsilon (\bw_0)&=&0, \qquad \text{in each}\ T,\\ 
\nabla \cdot\bw_0&=&0, \qquad \text{in each}\ T,\label{c2}\\
 \bw_0&=&\bw_b,  \qquad \text{on each}\ \partial T.\label{c3}
\end{eqnarray}

Since $\epsilon(\bw_0)=0$ on each element  $T$, it follows that $\bw_0\in RM(T)\subset [P_1(T)]^d$. Thus, $\bw_0=\bw_b$ on each $\partial T$,  and $\bw_0$ is continuous over $\Omega$. Given that  $\bw_b=0$ on $\partial\Omega$, we conclude $\bw_0=0$ on $\partial\Omega$. Applying the second Korn's inequality \eqref{korn}, we conclude $\bw_0\equiv 0$ in $\Omega$. Since   $\bw_0=\bw_b$ on each $\partial T$,  it follows that $\bw_b\equiv 0$ in $\Omega$. Hence, $\bu_h^{(1)}=\bu_h^{(2)}$. 

 This completes the proof of the theorem. 
\end{proof}

\section{Error Equations}
On each element $T\in\T_h$, let $Q_0$ denote the $L^2$ projection onto $P_k(T)$. On each  edge or face  $e\subset\partial T$, recall that $Q_b$ is the $L^2$ projection operator onto $S_k(e)$. For any $\bw\in [H^1(\Omega)]^d$, we define the $L^2$ projection into the weak finite element space $V_h$, denote by $Q_h \bw$, such that
$$
(Q_h\bw)|_T:=\{Q_0(\bw|_T),Q_b(\bw|_{\pT})\},\qquad \forall T\in\T_h.
$$
We also denote $Q_{r_1}$ and $Q_{r_2}$ as the $L^2$ projection operators onto the finite element spaces of piecewise polynomials of degrees $r_1$ and $r_2$,  respectively.

\begin{lemma}\cite{wangelas} \label{pror1}The following properties hold:
\begin{equation}\label{pro}
\epsilon_{w}(\bw) =Q_{r_1}\epsilon (\bw), \qquad \forall \bw\in [H^1(T)]^d,
\end{equation}
\begin{equation}\label{pro2}
\nabla_w\cdot \bw  =Q_{r_2}(\nabla \cdot  \bw), \qquad \forall \bw\in [H^1(T)]^d.
\end{equation}
\end{lemma}

Let  $\bu$ denote the exact solution of the elasticity interface problem \eqref{model}, and let $\bu_h \in V_{h}$ represent its numerical approximation obtained from the auto-stabilized WG Algorithm \ref{PDWG}. The error function $\be_h$ is defined as:
\begin{equation}\label{error} 
\be_h=\bu-\bu_h  \in V_{h}^0.
\end{equation}

\begin{lemma}\label{errorequa}
The error function $\be_h$ defined in  \eqref{error}  satisfies the following error equation:
\begin{equation}\label{erroreqn}
\sum_{T\in {\cal T}_h}(2\mu\epsilon_w (\be_h), \epsilon_w (\bv))_T+(\lambda \nabla_w \cdot \be_h, \nabla_w \cdot \bv)_T =\ell (\bu, \bv), \qquad \forall \bv\in V_h^0,
\end{equation}
where 
$$
\ell (\bu, \bv)=\sum_{T\in {\cal T}_h}  \langle   \bv_b-\bv_0, 2\mu( Q_{r_1}-I)  \epsilon (\bu) \cdot\bn\rangle_{\partial T} +\langle (\bv_b-\bv_0)\cdot\bn, \lambda   (Q_{r_2} -I)(\nabla  \cdot \bu)\rangle_{\partial T}.
$$
\end{lemma}
\begin{proof}   Using \eqref{pro} and \eqref{pro2}, and substituting $\bvarphi= Q_{r_1} \epsilon(\bu)$ in \eqref{2.5new} and $\phi=Q_{r_2} (\nabla\cdot\bu)$ in \eqref{disdivnew}, we obtain 
\begin{equation}\label{54}
\begin{split}
&\sum_{T\in {\cal T}_h}(2\mu\epsilon_w (\bu), \epsilon_w (\bv))_T+(\lambda \nabla_w \cdot \bu, \nabla_w \cdot \bv)_T\\=&\sum_{T\in {\cal T}_h} (2\mu Q_{r_1} \epsilon (\bu), \epsilon_w \bv)_T+(\lambda Q_{r_2} (\nabla  \cdot \bu), \nabla_w \cdot \bv)_T\\
=&\sum_{T\in {\cal T}_h} (2\mu\epsilon(\bv_0),   Q_{r_1} \epsilon (\bu))_T+ \langle   2\mu(\bv_b-\bv_0), \frac{1}{2}( Q_{r_1} \epsilon (\bu)+Q_{r_1} \epsilon (\bu)^T) \cdot\bn\rangle_{\partial T}\\&+  (\lambda \nabla\cdot \bv_0,  Q_{r_2} (\nabla  \cdot \bu))_T+ \langle \lambda (\bv_b-\bv_0)\cdot\bn,   Q_{r_2} (\nabla  \cdot \bu)\rangle_{\partial T}\\
=&\sum_{T\in {\cal T}_h} (2\mu\epsilon(\bv_0),     \epsilon (\bu))_T+ \langle   2\mu(\bv_b-\bv_0),  Q_{r_1}  \epsilon (\bu) \cdot\bn\rangle_{\partial T}\\&+   (\lambda \nabla\cdot \bv_0,     \nabla  \cdot \bu)_T+ \langle \lambda (\bv_b-\bv_0)\cdot\bn,    Q_{r_2} (\nabla  \cdot \bu)\rangle_{\partial T}\\
=&\sum_{T\in {\cal T}_h}(\bf,\bv_0) +\langle 2\mu \epsilon(\bu)\cdot\bn, \bv_0\rangle_{\partial T}+\langle \lambda \nabla\cdot\bu, \bv_0\cdot\bn\rangle_{\partial T}\\&+\langle   \bv_b-\bv_0, 2\mu Q_{r_1}  \epsilon (\bu) \cdot\bn\rangle_{\partial T} +\langle (\bv_b-\bv_0)\cdot\bn, \lambda   Q_{r_2} (\nabla  \cdot \bu)\rangle_{\partial T}\\
=&\sum_{T\in {\cal T}_h}(\bf,\bv_0)  +\langle   \bv_b-\bv_0, 2\mu( Q_{r_1}-I)  \epsilon (\bu) \cdot\bn\rangle_{\partial T} \\&+\langle (\bv_b-\bv_0)\cdot\bn, \lambda   (Q_{r_2} -I)(\nabla  \cdot \bu)\rangle_{\partial T}+\sum_{m=1}^M \langle \bv_b, \ljump2\mu\epsilon(\bu)\bn\rjump\rangle_{\Gamma_m}\\&+ \sum_{m=1}^M \langle \bv_b, \ljump\lambda\nabla\cdot\bu   \bn\rjump\rangle_{\Gamma_m}
\\=&\sum_{T\in {\cal T}_h}(\bf,\bv_0)  +\langle   \bv_b-\bv_0, 2\mu( Q_{r_1}-I)  \epsilon (\bu) \cdot\bn\rangle_{\partial T} \\&+\langle (\bv_b-\bv_0)\cdot\bn, \lambda   (Q_{r_2} -I)(\nabla  \cdot \bu)\rangle_{\partial T}+\sum_{m=1}^M \langle \bv_b, \bpsi_m   \rangle_{\Gamma_m} 
\end{split}
\end{equation}
where we used \eqref{model}, standard  integration by parts, and the following equalities: $\sum_{T\in {\cal T}_h} \langle 2\mu \epsilon(\bu)\cdot\bn, \bv_b\rangle_{\partial T} =  \sum_{T\in {\cal T}_h} \langle \bv_b, \ljump 2\mu\epsilon(\bu)\bn\rjump\rangle_{\pT}=\sum_{m=1}^M \langle \bv_b, \ljump 2\mu\epsilon(\bu)\bn\rjump\rangle_{\Gamma_m}$ and $\sum_{T\in {\cal T}_h}\langle   \lambda \nabla\cdot\bu, \bv_b\cdot\bn\rangle_{\partial T} = \sum_{T\in {\cal T}_h}\langle \bv_b, \ljump\lambda\nabla\cdot\bu  \bn\rjump\rangle_{\pT}=\sum_{m=1}^M \langle \bv_b, \ljump\lambda\nabla\cdot\bu   \bn\rjump\rangle_{\Gamma_m}$  since $\bv_b=0$ on $\partial \Omega$.

Finally, subtracting \eqref{WG} from \eqref{54}   completes the proof of the lemma.
\end{proof}

\section{Error Estimates}
\begin{lemma}\cite{wy3655}
Let ${\cal T}_h$ denote a finite element partition of the domain $\Omega$ that satisfies the shape regular assumption  outlined in  \cite{wy3655}. For any $0\leq s \leq 1$, $1\leq m \leq k+1$, and $0\leq n \leq k$, the following inequalities hold: 
\begin{eqnarray}\label{error1}
 \sum_{T\in {\cal T}_h}h_T^{2s}\|\epsilon (\bu)- Q_{r_1} \epsilon( \bu)\|^2_{s,T}&\leq& C  h^{2m-2}\|\bu\|^2_{m},\\
 \label{error3} \sum_{T\in {\cal T}_h}h_T^{2s}\|\nabla\cdot \bu- Q_{r_2} \nabla\cdot \bu\|^2_{s,T}&\leq& C  h^{2m-2}\|\bu\|^2_{m},\\
\label{error2}
\sum_{T\in {\cal T}_h}h_T^{2s}\|\bu- Q _0\bu\|^2_{s,T}&\leq& C h^{2n+2}\|\bu\|^2_{n+1}.
\end{eqnarray}
 \end{lemma}

\begin{lemma}
Assume $k\geq 1$ and the coefficients $\mu$ and $\lambda$ are piecewise constants with respect to the finite element partition ${\cal T}_h$. Let $\bu$  be the exact solution of the elasticity interface problem \eqref{model}. Assume $\bu$ is sufficiently regular such that $u\in \prod_{i=1}^N  [H^{k+1}(\Omega_i)]^d$. Then, the following error estimate holds: 
\begin{equation}\label{erroresti1}
\3bar \bu-Q_h\bu \3bar \leq  Ch^{k}\left( \sum_{i=1}^N \|\bu\|^2_{k+1, \Omega_i} \right)^{\frac{1}{2}}.
\end{equation}
\end{lemma}
\begin{proof}
Using the definitions of $\mu$ and $\lambda$ (the Lam$\acute{e}$ constants),  \eqref{2.5new}, the Cauchy-Schwarz inequality, the trace inequalities \eqref{tracein}-\eqref{trace},    and the estimate \eqref{error2} with $n=k$ and $s=0, 1$,  we obtain:
\begin{equation*}
\begin{split} 
&\quad\sum_{T\in {\cal T}_h}(2\mu\epsilon_w(\bu-Q_h\bu), \bvarphi)_T \\
&=\sum_{T\in {\cal T}_h} (2\mu\epsilon(\bu-Q_0\bu),  \bvarphi)_T+\langle 2\mu(Q_0\bu-Q_b\bu), \frac{1}{2}(\bvarphi+\bvarphi^T)\cdot\bn\rangle_{\partial T}\\
&\leq \Big(\sum_{T\in {\cal T}_h}\|2\mu\epsilon(\bu-Q_0\bu)\|^2_T\Big)^{\frac{1}{2}} \Big(\sum_{T\in {\cal T}_h}\|\bvarphi\|_T^2\Big)^{\frac{1}{2}}\\&\quad+ \Big(\sum_{T\in {\cal T}_h} \|2\mu(Q_0\bu-Q_b\bu)\|_{\partial T} ^2\Big)^{\frac{1}{2}}\Big(\sum_{T\in {\cal T}_h} \|\bvarphi\|_{\partial T}^2\Big)^{\frac{1}{2}}\\
&\leq C\Big(\ \sum_{T\in {\cal T}_h} \|\epsilon(\bu-Q_0\bu)\|_T^2\Big)^{\frac{1}{2}}\Big(\sum_{T\in {\cal T}_h} \|\bvarphi\|_T^2\Big)^{\frac{1}{2}}\\&\quad+C\Big(\sum_{T\in {\cal T}_h}h_T^{-1} \|Q_0\bu-\bu\|_{T} ^2+h_T \|Q_0\bu-\bu\|_{1,T} ^2\Big)^{\frac{1}{2}}\Big(\sum_{T\in {\cal T}_h} h_T^{-1}\|\bvarphi\|_T^2\Big)^{\frac{1}{2}}\\
&\leq Ch^k\left( \sum_{i=1}^N \|\bu\|^2_{k+1, \Omega_i} \right)\Big(\sum_{T\in {\cal T}_h} \|\bvarphi\|_T^2\Big)^{\frac{1}{2}},
\end{split}
\end{equation*}
 for any $\bvarphi\in [P_{r_1}(T)]^{d\times d}$. 
 
Letting $\bvarphi= \epsilon_w(\bu-Q_h\bu)$ yields 
\begin{equation}\label{n1}
\begin{split}
   &  \sum_{T\in {\cal T}_h}(2\mu\epsilon_w(\bu-Q_h\bu),  \epsilon_w(\bu-Q_h\bu))_T\\\leq & 
 Ch^k\left( \sum_{i=1}^N \|\bu\|^2_{k+1, \Omega_i} \right)\Big(\sum_{T\in {\cal T}_h} \|\epsilon_w(\bu-Q_h\bu)\|_T^2\Big)^{\frac{1}{2}}. 
\end{split}
\end{equation}

Using \eqref{disdivnew}, the Cauchy-Schwarz inequality, the trace inequalities \eqref{tracein}-\eqref{trace}, and the estimate \eqref{error2} with $n=k$ and $s=0, 1$,  we have

  \begin{equation*}
\begin{split}
&\quad\sum_{T\in {\cal T}_h}(\lambda \nabla_w \cdot(\bu-Q_h\bu), \phi)_T \\
&=\sum_{T\in {\cal T}_h} (\lambda \nabla\cdot(\bu-Q_0\bu),  \phi)_T+\langle \lambda(Q_0\bu-Q_b\bu)\cdot\bn,  \phi\rangle_{\partial T}\\
&\leq \Big(\sum_{T\in {\cal T}_h}\|\lambda \nabla\cdot(\bu-Q_0\bu)\|^2_T\Big)^{\frac{1}{2}} \Big(\sum_{T\in {\cal T}_h}\|\phi\|_T^2\Big)^{\frac{1}{2}}\\&\quad+ \Big(\sum_{T\in {\cal T}_h} \|\lambda(Q_0\bu-Q_b\bu)\cdot\bn\|_{\partial T} ^2\Big)^{\frac{1}{2}}\Big(\sum_{T\in {\cal T}_h} \|\phi\|_{\partial T}^2\Big)^{\frac{1}{2}}\\
&\leq C\Big(\ \sum_{T\in {\cal T}_h} \|\nabla\cdot(\bu-Q_0\bu)\|_T^2\Big)^{\frac{1}{2}}\Big(\sum_{T\in {\cal T}_h} \|\phi\|_T^2\Big)^{\frac{1}{2}}\\&\quad+C\Big(\sum_{T\in {\cal T}_h}h_T^{-1} \|Q_0\bu-\bu\|_{T} ^2+h_T \|Q_0\bu-\bu\|_{1,T} ^2\Big)^{\frac{1}{2}}\Big(\sum_{T\in {\cal T}_h}h_T^{-1}\|\phi\|_T^2\Big)^{\frac{1}{2}}\\
&\leq Ch^k\left( \sum_{i=1}^N \|\bu\|^2_{k+1, \Omega_i} \right)\Big(\sum_{T\in {\cal T}_h} \|\phi\|_T^2\Big)^{\frac{1}{2}},
\end{split}
\end{equation*}
 for any $\phi\in  P_{r_2}(T)$. 
 
Letting $\phi= \nabla_w \cdot(\bu-Q_h\bu)$ gives 
\begin{equation}\label{n2}
\begin{split}
  & \sum_{T\in {\cal T}_h}(\lambda \nabla_w \cdot(\bu-Q_h\bu), \nabla_w \cdot(\bu-Q_h\bu))_T\\ \leq &
 Ch^k\left( \sum_{i=1}^N \|\bu\|^2_{k+1, \Omega_i} \right)\Big(\sum_{T\in {\cal T}_h} \|\nabla_w \cdot(\bu-Q_h\bu)\|_T^2\Big)^{\frac{1}{2}}. 
\end{split}
\end{equation}

Combining \eqref{n1} and \eqref{n2}, we conclude that $$
 \3bar\bu-Q_h\bu\3bar ^2\leq 
 Ch^k\left( \sum_{i=1}^N \|\bu\|^2_{k+1, \Omega_i} \right)\Big(\Big(\sum_{T\in {\cal T}_h} \|\epsilon_w(\bu-Q_h\bu)\|_T^2\Big)^{\frac{1}{2}}+\Big(\sum_{T\in {\cal T}_h} \|\nabla_w \cdot(\bu-Q_h\bu)\|_T^2\Big)^{\frac{1}{2}}\Big).
 $$
 This completes the proof of the lemma.
\end{proof}

\begin{theorem}
Assume  the exact solution $\bu$ of the elasticity interface problem \eqref{model} is sufficiently regular such that    $u\in \prod_{i=1}^N  [H^{k+1}(\Omega_i)]^d$. Then, there exists a constant $C>0$, such that the following error estimate holds: 
\begin{equation}\label{trinorm}
\3bar \bu-\bu_h\3bar \leq Ch^k\left( \sum_{i=1}^N \|\bu\|^2_{k+1, \Omega_i} \right).
\end{equation}
\end{theorem}
\begin{proof}
For   the right-hand side of the error equation \eqref{erroreqn}, we apply the Cauchy-Schwarz inequality, the trace inequality \eqref{tracein},  the estimates \eqref{error1}-\eqref{error3} with $m=k+1$ and $s=0,1$, and \eqref{normeq},  to obtain   
\begin{equation}\label{erroreqn1}
\begin{split}
& \Big|\sum_{T\in {\cal T}_h}\langle   \bv_b-\bv_0, 2\mu( Q_{r_1}-I)  \epsilon (\bu) \cdot\bn\rangle_{\partial T} \\&+\langle (\bv_b-\bv_0)\cdot\bn, \lambda   (Q_{r_2} -I)(\nabla  \cdot \bu)\rangle_{\partial T}\Big|\\
\leq & C\Big(\sum_{T\in {\cal T}_h}\|2\mu( Q_{r_1}-I)  \epsilon (\bu) \cdot\bn\|^2_T+h_T^2\|2\mu( Q_{r_1}-I)  \epsilon (\bu) \cdot\bn\|^2_{1,T}\Big)^{\frac{1}{2}}  \\& \cdot\Big(\sum_{T\in {\cal T}_h}h_T^{-1}\|\bv_b-\bv_0\|^2_{\partial T}\Big)^{\frac{1}{2}}\\
&+C\Big(\sum_{T\in {\cal T}_h}\|\lambda(Q_{r_2} -I)(\nabla  \cdot \bu)\|^2_T+h_T^2\|\lambda(Q_{r_2} -I)(\nabla  \cdot \bu)\|^2_{1,T}\Big)^{\frac{1}{2}}  \\&\cdot\Big(\sum_{T\in {\cal T}_h}h_T^{-1}\|\bv_b-\bv_0\|^2_{\partial T}\Big)^{\frac{1}{2}}\\
\leq & Ch^k\left( \sum_{i=1}^N \|\bu\|^2_{k+1, \Omega_i} \right) \| \bv\|_{1,h}\\
\leq & Ch^k\left( \sum_{i=1}^N \|\bu\|^2_{k+1, \Omega_i} \right) \3bar \bv\3bar.
\end{split}
\end{equation}
Substituting \eqref{erroreqn1}  into \eqref{erroreqn}, we have
\begin{equation}\label{err}
\sum_{T\in {\cal T}_h}(2\mu\epsilon_w (\be_h), \epsilon_w (\bv))_T+(\lambda \nabla_w \cdot \be_h, \nabla_w \cdot \bv)_T\leq   Ch^k\left( \sum_{i=1}^N \|\bu\|^2_{k+1, \Omega_i} \right) \3bar  \bv\3bar.
\end{equation} 

Now, by  applying the  Cauchy-Schwarz inequality and  substituting  $\bv=Q_h\bu-\bu_h$  into \eqref{err}, along with  the estimate \eqref{erroresti1}, we obtain
\begin{equation*}
\begin{split}
& \3bar \bu-\bu_h\3bar^2\\=&\sum_{T\in {\cal T}_h}(2\mu\epsilon_w (\bu-\bu_h), \epsilon_w (\bu-Q_h\bu))_T+(2\mu\epsilon_w (\bu-\bu_h), \epsilon_w (Q_h\bu-\bu_h))_T\\
&+(\lambda\nabla_w \cdot(\bu-\bu_h), \nabla_w \cdot (\bu-Q_h\bu))_T+(\lambda\nabla_w \cdot (\bu-\bu_h), \nabla_w \cdot (Q_h\bu-\bu_h))_T\\
\leq &\Big(\sum_{T\in {\cal T}_h}2\mu\|\epsilon_w (\bu-\bu_h)\|^2_T\Big)^{\frac{1}{2}} \Big(\sum_{T\in {\cal T}_h}2\mu\| \epsilon_w(\bu-Q_h\bu)\|^2_T\Big)^{\frac{1}{2}} \\
&+ \Big(\sum_{T\in {\cal T}_h}\lambda\|\nabla_w \cdot(\bu-\bu_h)\|^2_T\Big)^{\frac{1}{2}} \Big(\sum_{T\in {\cal T}_h}\lambda\| \nabla_w \cdot(\bu-Q_h\bu)\|^2_T\Big)^{\frac{1}{2}} \\& +\sum_{T\in {\cal T}_h} (2\mu\epsilon_w (\bu-\bu_h), \epsilon_w (Q_h\bu-\bu_h))_T+(\lambda\nabla_w \cdot (\bu-\bu_h), \nabla_w \cdot (Q_h\bu-\bu_h))_T\\
\leq &\3bar \bu-\bu_h \3bar  \3bar \bu-Q_h\bu \3bar+ Ch^k\left( \sum_{i=1}^N \|\bu\|^2_{k+1, \Omega_i} \right) \3bar Q_h\bu-\bu_h\3bar\\
\leq &\3bar \bu-\bu_h  \3bar  h^k\left( \sum_{i=1}^N \|\bu\|^2_{k+1, \Omega_i} \right) + Ch^k\left( \sum_{i=1}^N \|\bu\|^2_{k+1, \Omega_i} \right)  (\3bar Q_h\bu-\bu\3bar+\3bar \bu-\bu_h \3bar)  \\
\leq &\3bar \bu-\bu_h  \3bar  h^k\left( \sum_{i=1}^N \|\bu\|^2_{k+1, \Omega_i} \right) + Ch^k\left( \sum_{i=1}^N \|\bu\|^2_{k+1, \Omega_i} \right)   h^k\left( \sum_{i=1}^N \|\bu\|^2_{k+1, \Omega_i} \right)\\&+Ch^k\left( \sum_{i=1}^N \|\bu\|^2_{k+1, \Omega_i} \right) \3bar \bu-\bu_h\3bar.
\end{split}
\end{equation*}
This simplifies to 
\begin{equation*}
\begin{split}
 \3bar \bu-\bu_h\3bar  \leq Ch^k\left( \sum_{i=1}^N \|\bu\|^2_{k+1, \Omega_i} \right).
\end{split}
\end{equation*} 

Thus, the proof of the theorem is complete.
\end{proof}
 
\section{Numerical test}

We solve an interface elasticity problem \eqref{model} on two subdomains in two cases, $\Omega=(0,1)^2$:  
\an{ \label{o1}&& \Omega_0&=(\frac 12, 1)\times (0,1), &\quad \Omega_1&=\Omega\setminus \Omega_0, \\
     && \Omega_0&=(\frac 14, \frac 34)\times (\frac 14, \frac 34),
         &\quad \Omega_1&=\Omega\setminus \Omega_0.  && \label{o2} }
         
\subsection{Example \ref{ex-1}}\label{ex-1} In order to find the interface singularity,  we compute the
 first example without knowing its exact solution.
 Let $\sigma$ and $\lambda$ in \eqref{model} be defined by
 \an{ \label{l-1} 2\sigma=\lambda= \begin{cases} 1,
               & \t{ in } \ \Omega\setminus\Omega_0, \\
    10, & \t{ in } \  \Omega_0, 
        \end{cases}
    } where where $\Omega$ and $\Omega_0$ are defined in \eqref{o1}.  
    The right-hand side function is the simplest one,
\an{\label{f-1} \b f = \p{1\\-1}, \quad \ \t{on } \ \Omega. }

We solve the interface problem \eqref{model} with the coefficients in \eqref{l-1} and the
    $\b f$ in \eqref{f-1}  by the $P_3$ WG finite element on the (non-convex polygonal) grid $G_3$ shown
     in Figure \ref{g2-1}.
The solution is plotted in Figure \ref{gs-1}.

\begin{figure}[H] \centering
  \begin{picture}(420,100)(0,20)
  \put(0,-455){\includegraphics[width=410pt]{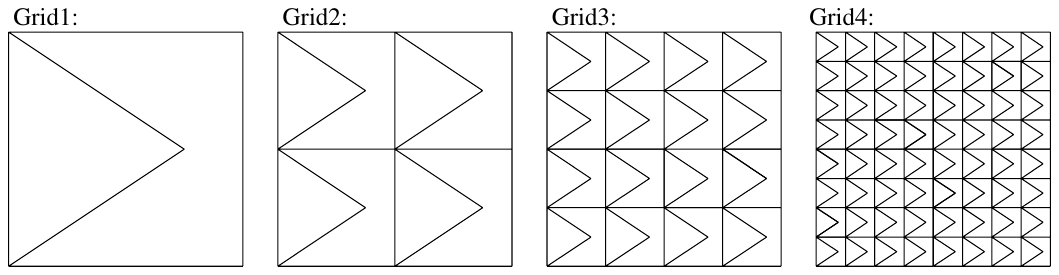}}
  \end{picture}
  \caption{Type-1 non-convex polygonal grids: first four grids $G_1$--$G_4$. } \label{g2-1}
	\end{figure}

We can see, from Figure \ref{gs-1},
   that the solution is continuous and nonzero at the interface and that
  the normal stress jumps at the interface.
We would use these two restrictions to design our next example.

\begin{figure}[H] \centering
  \begin{picture}(320,80)(0,0)
  \put(-25,-215){\includegraphics[width=3.1in]{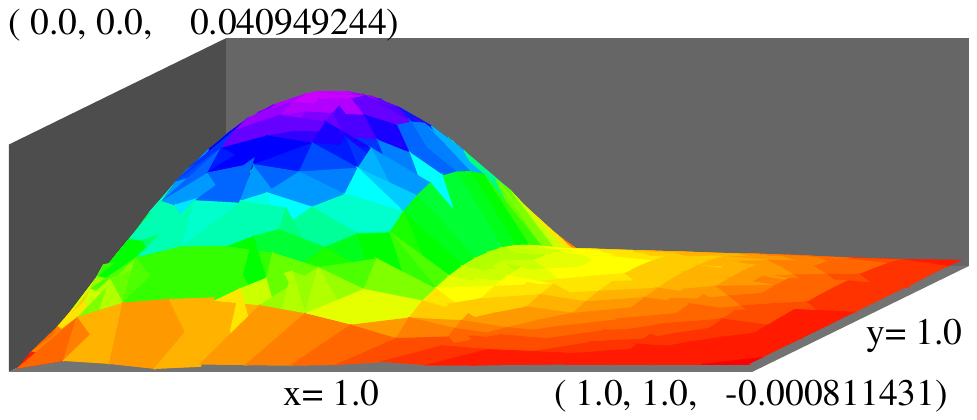}}
  \put(152,-215){\includegraphics[width=3.1in]{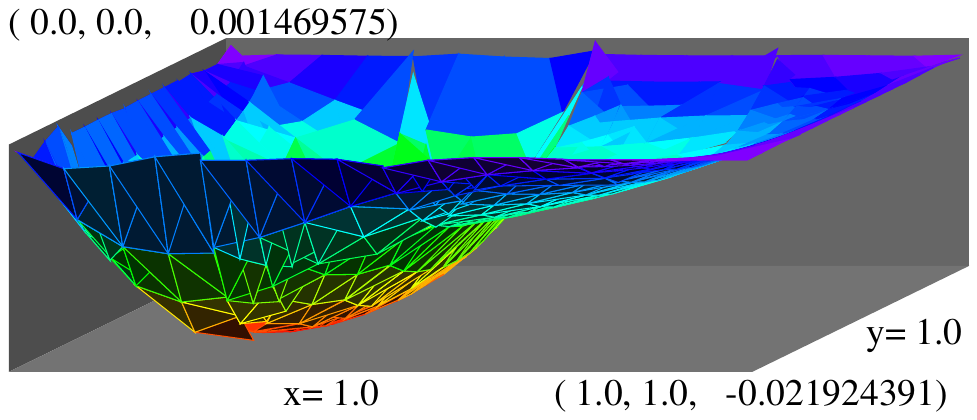}}
  \end{picture}
  \caption{The $(\b u_h)_1$ and $(\b u_h)_2$ of the $P_3$ WG solution
       for \eqref{model} with \eqref{l-1} and \eqref{f-1}
     on Grid 3 in Figure \ref{g2-1}. } \label{gs-1}
	\end{figure}

\subsection{Example \ref{ex-2}}\label{ex-2} 
We choose the coefficients and the exact solution of \eqref{model} as 
 \an{ \label{2-1} 2\sigma&=\lambda= \begin{cases} 1,
               & \t{ in } \ \Omega\setminus\Omega_0, \\
    1, & \t{ in } \  \Omega_0, 
        \end{cases} \\
  \label{2-s1} \b u&=\begin{cases} 2 x^2(1-x)y\p{ 4 x y - 16 y^2 - 3x + 8 y + 1\\
                                       4 x y - \,\; 8 y^2 - 3 x + 4 y + 1} ,
    & \t{ in } \ \Omega\setminus\Omega_0, \\
    -2 x^2(1-x)(1-y)\p{ 8 y^2 - 4 y  + x - 1  \\ \qquad\qquad \quad
       x-1}, & \t{ in } \  \Omega_0, 
        \end{cases}
    } where where $\Omega$ and $\Omega_0$ are defined in \eqref{o1}.  
    
This problem is smooth.  We need its data for comparisons with the singular solutions.
We compute the solution \eqref{2-s1} by three WG finite element methods where
  the weak symmetric gradient and the weak divergence are computed 
    with $r_1=r_2=k+1$ and $r_1=r_2=k+1$ in \eqref{w-d} on the triangles and the non-convex polygons
   in Figure \ref{g2-1} respectively.
The results are listed in Table \ref{t2-1} where we can see all solutions converge at the
  corresponding optimal orders.
  In fact, it seems that $P_2$ and $P_3$ WG methods converge at super orders.
 We do not have a theory for this superconvergence.   
 We would guess that it is because half of the polygons on each grid are triangles while
  some triangular WG methods are superconvergent.

  \begin{table}[H]
  \caption{ Error profile for computing \eqref{2-s1} on Figure \ref{g2-1} meshes.} \label{t2-1}
\begin{center}  
   \begin{tabular}{c|rr|rr}  
 \hline 
$G_i$ & 
       $\|\b u-\b u_h\|_{0}$&$O(h^r)$  &  $\|\nabla_w(\b u-\b u_h)\|_0$ & $O(h^r)$  \\ \hline  
     & \multicolumn{4}{c}{By the $P_1$ \eqref{Vk} WG element}  \\ 
 6&    0.499E-03 &  1.9&    0.356E-01 &  1.1 \\
 7&    0.126E-03 &  2.0&    0.177E-01 &  1.0 \\
 8&    0.315E-04 &  2.0&    0.885E-02 &  1.0 \\
 \hline 
     & \multicolumn{4}{c}{By the $P_2$ \eqref{Vk} WG element}  \\  
 4&    0.250E-03 &  4.0&    0.333E-01 &  3.1 \\
 5&    0.167E-04 &  3.9&    0.410E-02 &  3.0 \\
 6&    0.135E-05 &  3.6&    0.543E-03 &  2.9 \\
 \hline 
          & \multicolumn{4}{c}{By the $P_3$ \eqref{Vk} WG element}  \\ 
 4&    0.350E-04 &  5.1&    0.747E-02 &  4.1 \\
 5&    0.107E-05 &  5.0&    0.457E-03 &  4.0 \\
 6&    0.333E-07 &  5.0&    0.283E-04 &  4.0 \\
 \hline 
\end{tabular} \end{center}  \end{table}
 
We next recompute the solution \eqref{2-s1} on more non-convex polygonal grids, 
 shown in Figure \ref{g2-2}. 
 Here the weak symmetric gradient and the weak divergence are computed 
    with $r_1=r_2=k+1$ and $r_1=r_2=k+2$ in \eqref{w-d} on the 5-edge non-convex polygons and the 
      7-edge non-convex polygons
   in Figure \ref{g2-2} respectively.
The results are listed in Table \ref{t2-2} where we can see all solutions converge at the
  corresponding optimal orders.
Unlike the last computation on the other grids, there is no more superconvergence phenomenon.
 
\begin{figure}[H] \centering
  \begin{picture}(420,100)(0,20)
  \put(0,-455){\includegraphics[width=410pt]{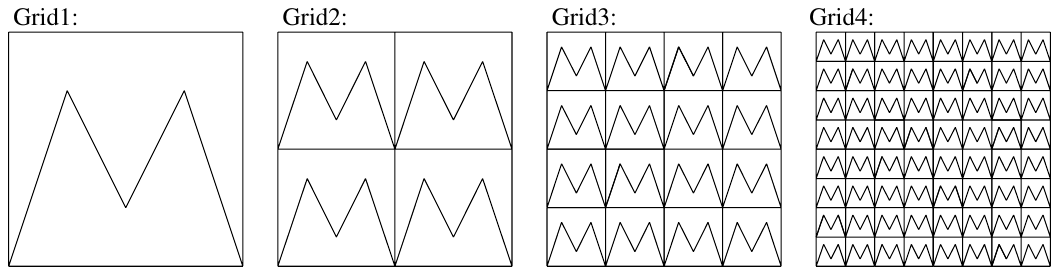}}
  \end{picture}
  \caption{Type-2 non-convex polygonal grids: first four grids $G_1$--$G_4$. } \label{g2-2}
	\end{figure}

  \begin{table}[H]
  \caption{ Error profile for computing \eqref{2-s1} on Figure \ref{g2-2} meshes.} \label{t2-2}
\begin{center}  
   \begin{tabular}{c|rr|rr}  
 \hline 
$G_i$ & 
       $\|\b u-\b u_h\|_{0}$&$O(h^r)$  &  $\|\nabla_w(\b u-\b u_h)\|_0$ & $O(h^r)$  \\ \hline  
     & \multicolumn{4}{c}{By the $P_1$ \eqref{Vk} WG element}  \\ 
 5&    0.747E-03 &  1.9&    0.848E-01 &  1.0 \\
 6&    0.189E-03 &  2.0&    0.427E-01 &  1.0 \\
 7&    0.472E-04 &  2.0&    0.214E-01 &  1.0 \\
 \hline 
     & \multicolumn{4}{c}{By the $P_2$ \eqref{Vk} WG element}  \\  
 4&    0.361E-03 &  2.8&    0.283E-01 &  2.0 \\
 5&    0.467E-04 &  3.0&    0.699E-02 &  2.0 \\
 6&    0.590E-05 &  3.0&    0.174E-02 &  2.0 \\
 \hline 
          & \multicolumn{4}{c}{By the $P_3$ \eqref{Vk} WG element}  \\ 
 3&    0.565E-03 &  3.8&    0.360E-01 &  3.6 \\
 4&    0.329E-04 &  4.1&    0.303E-02 &  3.6 \\
 5&    0.190E-05 &  4.1&    0.310E-03 &  3.3 \\
 \hline 
\end{tabular} \end{center}  \end{table}

\subsection{Example \ref{ex-3}}\label{ex-3} 
We choose the coefficients and the exact solution of \eqref{model} as 
 \an{ \label{2-2} 2\sigma&=\lambda= \begin{cases} 1,
               & \t{ in } \ \Omega\setminus\Omega_0, \\
    10, & \t{ in } \  \Omega_0, 
        \end{cases} \\
  \label{2-s2}   \b u&=\begin{cases} 2 x (1-x)y\b u_1 ,
    & \t{ in } \ \Omega\setminus\Omega_0, \\
    -2 x^2(1-x)(1-y)\p{ 8 y^2 - 4 y  + x - 1  \\ \qquad\qquad \quad
       x-1}, & \t{ in } \  \Omega_0, 
        \end{cases}
    } where where $\Omega$ and $\Omega_0$ are defined in \eqref{o1},
     and 
     \a{ \b u_1 &= \p{ 22 x^2 y - 196 x y^2 + 72 y^3 - 12 x^2 + 98 x y - 36 y^2 + x\\
     22 x^2 y - \,\; 80 x y^2 + 36 y^3 - 12 x^2 + 40 x y - 18 y^2 + x  }.
                      }
We plot the $P_1$ WG solution for \eqref{2-2}--\eqref{2-s2} on the fifth grid $G_5$ shown
   in Figure \ref{g2-1}, in Figure \ref{gs-3}.

\begin{figure}[H] \centering
  \begin{picture}(320,80)(0,0)
  \put(-25,-215){\includegraphics[width=3.1in]{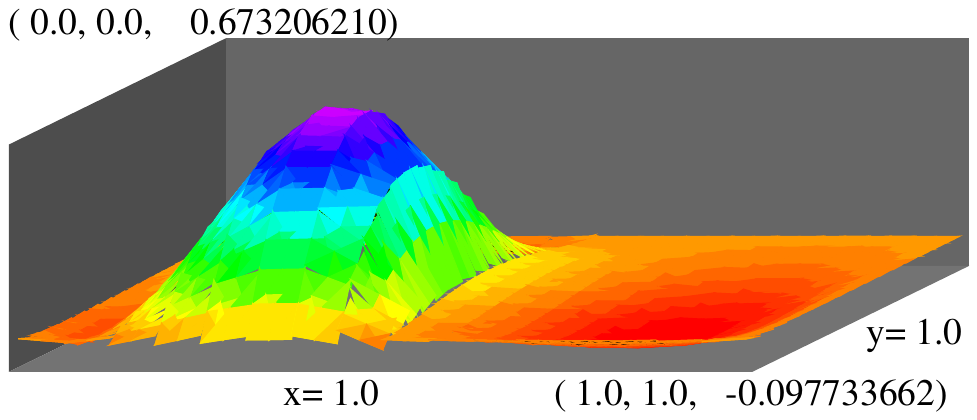}}
  \put(152,-215){\includegraphics[width=3.1in]{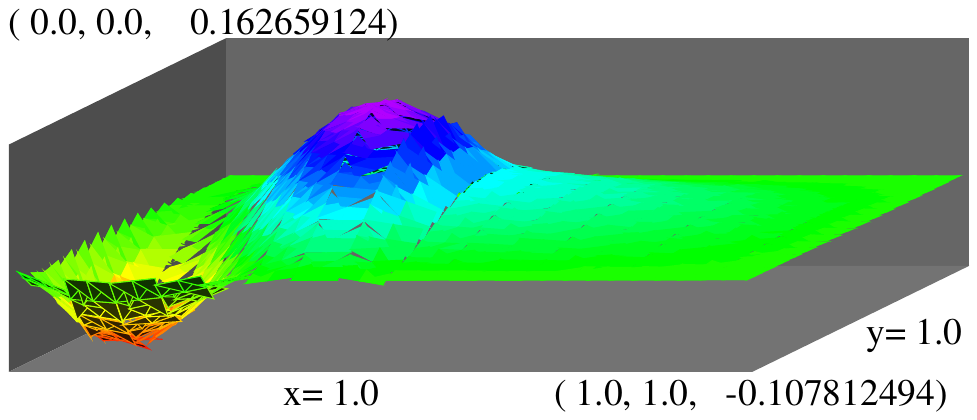}}
  \end{picture}
  \caption{The $(\b u_h)_1$ and $(\b u_h)_2$ of the $P_1$ WG solution
       for \eqref{2-2}--\eqref{2-s2}
     on Grid 5 in Figure \ref{g2-1}. } \label{gs-3}
	\end{figure}

As in Example \ref{ex-2}, we compute the solution \eqref{2-s2} by three WG finite element methods with
 the same  $r_1$ and $r_2$,  on two types of non-convex polygonal grids, shown in Figures \ref{g2-1}
    and \ref{g2-2}.
The results are listed in Tables \ref{t3-1} and \ref{t3-2}  
   where we can see all solutions converge at the
  corresponding optimal orders. 
  
  \begin{table}[H]
  \caption{ Error profile for computing \eqref{2-s2} on Figure \ref{g2-1} meshes.} \label{t3-1}
\begin{center}  
   \begin{tabular}{c|rr|rr}  
 \hline 
$G_i$ & 
       $\|\b u-\b u_h\|_{0}$&$O(h^r)$  &  $\|\nabla_w(\b u-\b u_h)\|_0$ & $O(h^r)$  \\ \hline  
     & \multicolumn{4}{c}{By the $P_1$ \eqref{Vk} WG element}  \\ 
 5&    0.202E-02 &  2.3&    0.238E+00 &  1.6 \\
 6&    0.462E-03 &  2.1&    0.992E-01 &  1.3 \\
 7&    0.112E-03 &  2.0&    0.466E-01 &  1.1 \\
 \hline 
     & \multicolumn{4}{c}{By the $P_2$ \eqref{Vk} WG element}  \\  
 4&    0.230E-02 &  4.0&    0.298E+00 &  3.0 \\
 5&    0.154E-03 &  3.9&    0.374E-01 &  3.0 \\
 6&    0.122E-04 &  3.7&    0.493E-02 &  2.9 \\
 \hline 
          & \multicolumn{4}{c}{By the $P_3$ \eqref{Vk} WG element}  \\ 
 4&    0.399E-03 &  5.0&    0.822E-01 &  4.0 \\
 5&    0.123E-04 &  5.0&    0.508E-02 &  4.0 \\
 6&    0.384E-06 &  5.0&    0.316E-03 &  4.0 \\
 \hline 
\end{tabular} \end{center}  \end{table}

  \begin{table}[H]
  \caption{ Error profile for computing \eqref{2-s2} on Figure \ref{g2-2} meshes.} \label{t3-2}
\begin{center}  
   \begin{tabular}{c|rr|rr}  
 \hline 
$G_i$ & 
       $\|\b u-\b u_h\|_{0}$&$O(h^r)$  &  $\|\nabla_w(\b u-\b u_h)\|_0$ & $O(h^r)$  \\ \hline  
     & \multicolumn{4}{c}{By the $P_1$ \eqref{Vk} WG element}  \\ 
 5&    0.416E-02 &  1.9&    0.451E+00 &  0.9 \\
 6&    0.106E-02 &  2.0&    0.228E+00 &  1.0 \\
 7&    0.265E-03 &  2.0&    0.115E+00 &  1.0 \\
 \hline 
     & \multicolumn{4}{c}{By the $P_2$ \eqref{Vk} WG element}  \\  
 4&    0.273E-02 &  2.8&    0.196E+00 &  2.0 \\
 5&    0.357E-03 &  2.9&    0.483E-01 &  2.0 \\
 6&    0.452E-04 &  3.0&    0.121E-01 &  2.0 \\
 \hline 
          & \multicolumn{4}{c}{By the $P_3$ \eqref{Vk} WG element}  \\ 
 3&    0.483E-02 &  3.7&    0.326E+00 &  3.5 \\
 4&    0.289E-03 &  4.1&    0.271E-01 &  3.6 \\
 5&    0.170E-04 &  4.1&    0.271E-02 &  3.3 \\
 \hline 
\end{tabular} \end{center}  \end{table}

\subsection{Example \ref{ex-4}}\label{ex-4} 
We choose the coefficients and the exact solution of \eqref{model} as 
 \an{ \label{3-1} 2\sigma&=\lambda= \begin{cases} 1,
               & \t{ in } \ \Omega\setminus\Omega_0, \\
    1, & \t{ in } \  \Omega_0, 
        \end{cases} \\
  \label{3-s1}  \b u&=\begin{cases} \frac{2^8}{33} x(1-x) y (1-y)B_4 \p{1\\-4} ,
    & \t{ in } \ \Omega\setminus\Omega_0, \\
    \frac {2^{12}}{11}  
        (  x^2 +   y^2 -  x -  y + \frac 3{16}) B_4 \p{-1\\4}, & \t{ in } \  \Omega_0, 
        \end{cases}
    } where where $\Omega$ and $\Omega_0$ are defined in \eqref{o2} and 
\an{\label{b4} B_4 =  (  x - \frac 14)( x -\frac  3 4)( y - \frac 14) ( y -\frac  34). 
    } 
     
We plot the $P_1$ WG solution for \eqref{3-1}--\eqref{3-s1} on the sixth grid $G_6$ shown
   in Figure \ref{g2-1}, in Figure \ref{gs-4}.

\begin{figure}[H] \centering
  \begin{picture}(320,80)(0,0)
  \put(-25,-215){\includegraphics[width=3.1in]{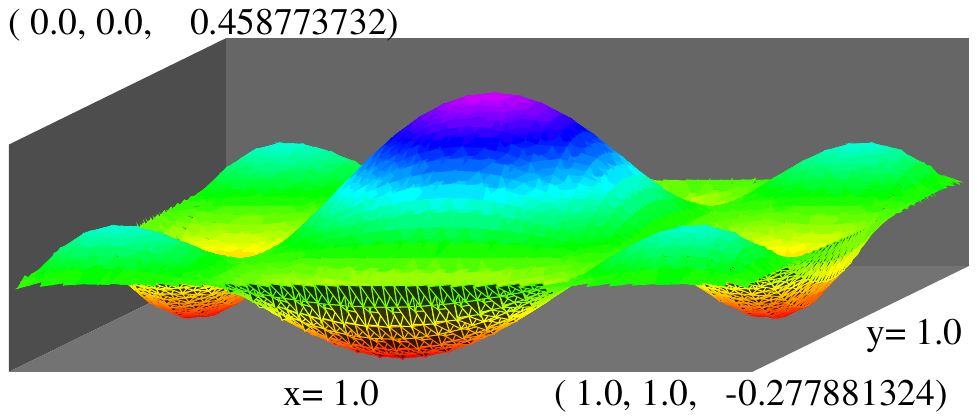}}
  \put(152,-215){\includegraphics[width=3.1in]{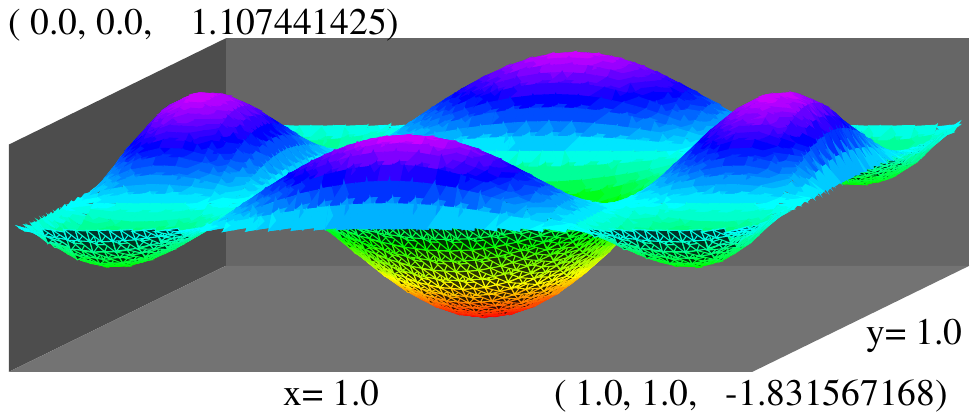}}
  \end{picture}
  \caption{The $(\b u_h)_1$ and $(\b u_h)_2$ of the $P_1$ WG solution
       for \eqref{3-1}--\eqref{3-s1}
     on Grid 6 in Figure \ref{g2-1}. } \label{gs-4}
	\end{figure}

As in the last two examples, we compute the solution \eqref{3-s1} by three WG finite element methods with
 the same  $r_1$ and $r_2$,  on two types of non-convex polygonal grids, shown in Figures \ref{g2-1}
    and \ref{g2-2}.
The results are listed in Tables \ref{t4-1} and \ref{t4-2}  
   where we can see all solutions converge at the
  corresponding optimal orders. 
  
  \begin{table}[H]
  \caption{ Error profile for computing \eqref{3-s1} on Figure \ref{g2-1} meshes.} \label{t4-1}
\begin{center}  
   \begin{tabular}{c|rr|rr}  
 \hline 
$G_i$ & 
       $\|\b u-\b u_h\|_{0}$&$O(h^r)$  &  $\|\nabla_w(\b u-\b u_h)\|_0$ & $O(h^r)$  \\ \hline  
     & \multicolumn{4}{c}{By the $P_1$ \eqref{Vk} WG element}  \\ 
 5&    0.791E-02 &  2.8&    0.121E+01 &  1.8 \\
 6&    0.135E-02 &  2.6&    0.415E+00 &  1.5 \\
 7&    0.283E-03 &  2.3&    0.176E+00 &  1.2 \\
 \hline 
     & \multicolumn{4}{c}{By the $P_2$ \eqref{Vk} WG element}  \\  
 4&    0.233E-01 &  4.2&    0.315E+01 &  2.9 \\
 5&    0.143E-02 &  4.0&    0.401E+00 &  3.0 \\
 6&    0.905E-04 &  4.0&    0.509E-01 &  3.0 \\ 
 \hline 
          & \multicolumn{4}{c}{By the $P_3$ \eqref{Vk} WG element}  \\ 
 4&    0.577E-02 &  5.1&    0.128E+01 &  4.0 \\
 5&    0.179E-03 &  5.0&    0.805E-01 &  4.0 \\
 6&    0.557E-05 &  5.0&    0.504E-02 &  4.0 \\ 
 \hline 
\end{tabular} \end{center}  \end{table}
   
Comparing the results with those in Example \ref{ex-2},  on
  not so bad type-1 grids, the error of solutions for \eqref{3-s1} is not much worse than
    that for \eqref{2-s1}.
But the solution is much worse on bad polygonal grids of Figure \ref{g2-2},
  shown in Tables \ref{t4-2} and \ref{t2-2}.

  \begin{table}[H]
  \caption{ Error profile for computing \eqref{3-s1} on Figure \ref{g2-2} meshes.} \label{t4-2}
\begin{center}  
   \begin{tabular}{c|rr|rr}  
 \hline 
$G_i$ & 
       $\|\b u-\b u_h\|_{0}$&$O(h^r)$  &  $\|\nabla_w(\b u-\b u_h)\|_0$ & $O(h^r)$  \\ \hline  
     & \multicolumn{4}{c}{By the $P_1$ \eqref{Vk} WG element}  \\ 
 5&    0.356E-01 &  1.7&    0.268E+01 &  1.0 \\
 6&    0.948E-02 &  1.9&    0.135E+01 &  1.0 \\
 7&    0.239E-02 &  2.0&    0.676E+00 &  1.0 \\ 
 \hline 
     & \multicolumn{4}{c}{By the $P_2$ \eqref{Vk} WG element}  \\  
 4&    0.252E-01 &  2.8&    0.167E+01 &  2.1 \\
 5&    0.352E-02 &  2.8&    0.396E+00 &  2.1 \\
 6&    0.446E-03 &  3.0&    0.977E-01 &  2.0 \\ 
 \hline 
          & \multicolumn{4}{c}{By the $P_3$ \eqref{Vk} WG element}  \\ 
 4&    0.539E-02 &  4.2&    0.364E+00 &  3.6 \\
 5&    0.314E-03 &  4.1&    0.348E-01 &  3.4 \\
 6&    0.186E-04 &  4.1&    0.391E-02 &  3.2 \\
 \hline 
\end{tabular} \end{center}  \end{table}

\subsection{Example \ref{ex-5}}\label{ex-5} 
We choose the coefficients and the exact solution of \eqref{model} as 
 \an{ \label{3-2} 2\sigma&=\lambda= \begin{cases} 1,
               & \t{ in } \ \Omega\setminus\Omega_0, \\
    10^{-1}, & \t{ in } \  \Omega_0, 
        \end{cases} \\
  \label{3-s2}   \b u&=\begin{cases} \frac{2^{11}}{165}  x(1-x) y (1-y)B_4  \p{1\\-4},
    & \t{ in } \ \Omega\setminus\Omega_0, \\
    \frac {2^{12}}{11}  
        (  x^2 +   y^2 -  x -  y + \frac 3{16}) B_4 \p{-1\\4}, & \t{ in } \  \Omega_0, 
        \end{cases}
    } where $\Omega$ and $\Omega_0$ are defined in \eqref{o2} and $B_4$ is defined in \eqref{b4}.

We plot the $P_1$ WG solution for \eqref{3-2}--\eqref{3-s2} on the sixth grid $G_6$ shown
   in Figure \ref{g2-1}, in Figure \ref{gs-5}.

\begin{figure}[H] \centering
  \begin{picture}(320,80)(0,0)
  \put(-25,-215){\includegraphics[width=3.1in]{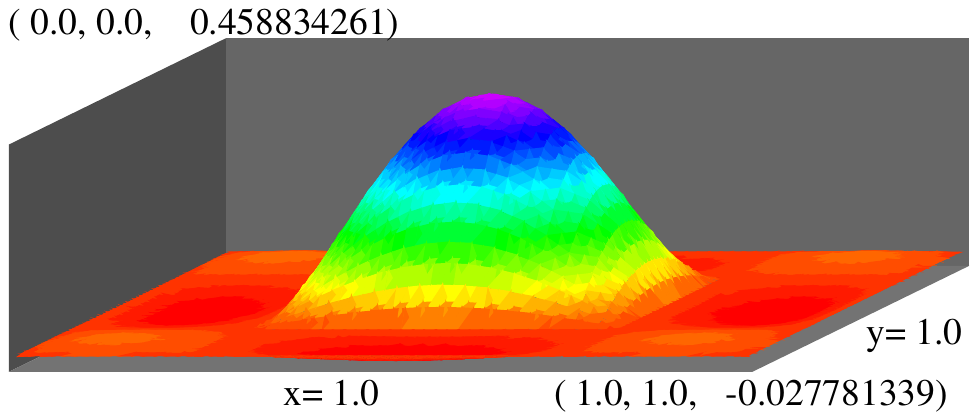}}
  \put(152,-215){\includegraphics[width=3.1in]{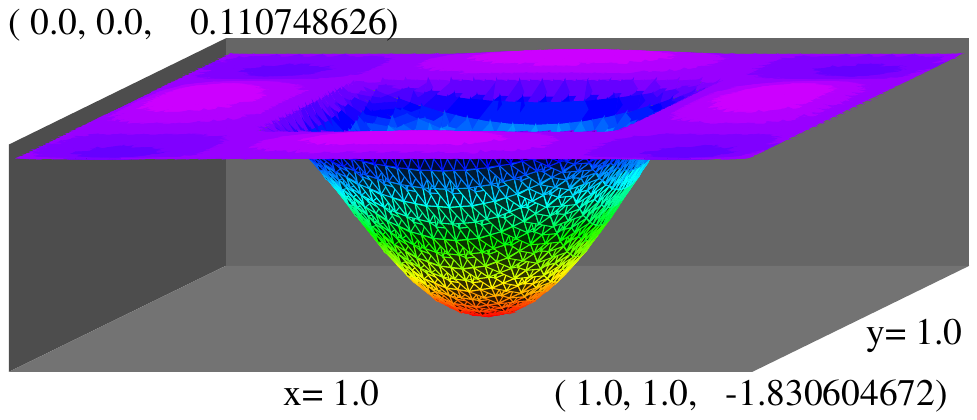}}
  \end{picture}
  \caption{The $(\b u_h)_1$ and $(\b u_h)_2$ of the $P_1$ WG solution
       for \eqref{3-2}--\eqref{3-s2}
     on Grid 6 in Figure \ref{g2-1}. } \label{gs-5}
	\end{figure}

Again we compute the solution \eqref{3-s2} by three WG finite element methods with
 the same  $r_1$ and $r_2$ as in previous examples,
   on two types of non-convex polygonal grids, shown in Figures \ref{g2-1}
    and \ref{g2-2}.
The results are listed in Table \ref{t5-1}, 
   where we can see all solutions converge at the
  corresponding optimal orders. 
Again the solutions on type-1 grids are much more accurate.

  \begin{table}[H]
  \caption{ Error profile for computing \eqref{3-s2}.} \label{t5-1}
\begin{center}  
   \begin{tabular}{c|rr|rr}  
 \hline 
$G_i$ & 
       $\|\b u-\b u_h\|_{0}$&$O(h^r)$  &  $\|\nabla_w(\b u-\b u_h)\|_0$ & $O(h^r)$  \\ \hline  
     & \multicolumn{4}{c}{By the $P_1$ \eqref{Vk} WG element on Figure \ref{g2-1} grids}  \\
 6&    0.362E-03 &  2.5&    0.127E+00 &  1.5 \\
 7&    0.766E-04 &  2.2&    0.561E-01 &  1.2 \\  \hline  
     & \multicolumn{4}{c}{By the $P_1$ \eqref{Vk} WG element on Figure \ref{g2-2} grids}  \\ 
 6&    0.308E-02 &  2.0&    0.470E+00 &  1.0 \\
 7&    0.767E-03 &  2.0&    0.236E+00 &  1.0 \\
 \hline 
     & \multicolumn{4}{c}{By the $P_2$ \eqref{Vk} WG element on Figure \ref{g2-1} grids}  \\  
 5&    0.267E-03 &  4.0&    0.775E-01 &  3.0 \\
 6&    0.171E-04 &  4.0&    0.998E-02 &  3.0 \\
 \hline 
     & \multicolumn{4}{c}{By the $P_2$ \eqref{Vk} WG element on Figure \ref{g2-2} grids}  \\  
 5&    0.714E-03 &  2.9&    0.931E-01 &  2.1 \\
 6&    0.900E-04 &  3.0&    0.229E-01 &  2.0 \\
 \hline 
          & \multicolumn{4}{c}{By the $P_3$ \eqref{Vk} WG element on Figure \ref{g2-1} grids}  \\ 
 5&    0.176E-01 &  5.0&    0.797E+01 &  4.0 \\
 6&    0.548E-03 &  5.0&    0.499E+00 &  4.0 \\
 \hline 
          & \multicolumn{4}{c}{By the $P_3$ \eqref{Vk} WG element on Figure \ref{g2-2} grids}  \\ 
 4&    0.797E-03 &  4.4&    0.633E-01 &  3.6 \\
 5&    0.470E-04 &  4.1&    0.590E-02 &  3.4 \\
 \hline 
\end{tabular} \end{center}  \end{table}

\subsection{Example \ref{ex-6}}\label{ex-6} 
We choose the coefficients and the exact solution of \eqref{model} as 
 \an{ \label{3-3} 2\sigma&=\lambda= \begin{cases} 1,
               & \t{ in } \ \Omega\setminus\Omega_0, \\
    10^{2}, & \t{ in } \  \Omega_0, 
        \end{cases} \\
  \label{3-s3} \b u&=\begin{cases} \frac{2^{10} 5^2 }{33}  x(1-x) y (1-y)B_4  \p{1\\-4},
    & \t{ in } \ \Omega\setminus\Omega_0, \\
    \frac {2^{12}}{11}  
        (  x^2 +   y^2 -  x -  y + \frac 3{16}) B_4 \p{-1\\4}, & \t{ in } \  \Omega_0, 
        \end{cases}
    } where $\Omega$ and $\Omega_0$ are defined in \eqref{o2} and $B_4$ is defined in \eqref{b4}.

We plot the $P_2$ WG solution for \eqref{3-3}--\eqref{3-s3} on grid $G_5$ shown
   in Figure \ref{g2-1}, in Figure \ref{gs-6}.

\begin{figure}[H] \centering
  \begin{picture}(320,80)(0,0)
  \put(-25,-215){\includegraphics[width=3.1in]{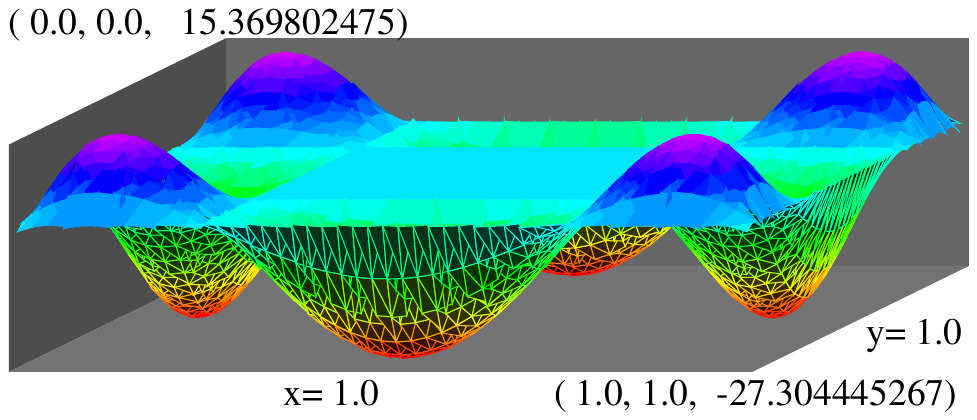}}
  \put(152,-215){\includegraphics[width=3.1in]{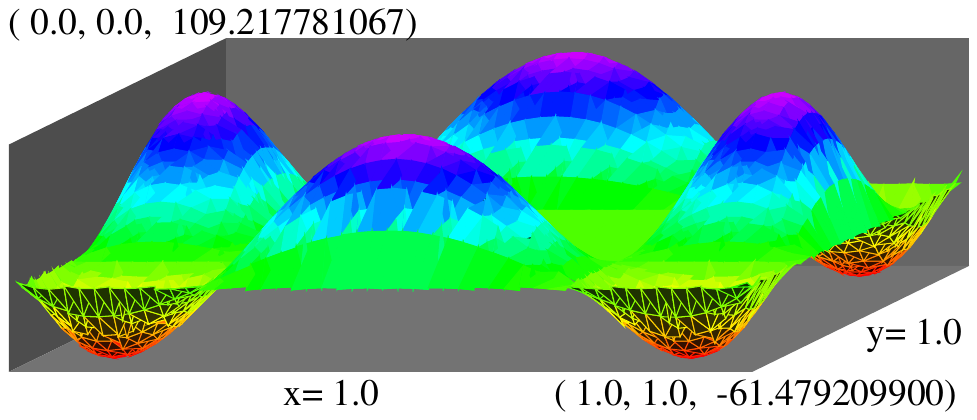}}
  \end{picture}
  \caption{The $(\b u_h)_1$ and $(\b u_h)_2$ of the $P_2$ WG solution
       for \eqref{3-3}--\eqref{3-s3}
     on Grid 5 in Figure \ref{g2-1}. } \label{gs-6}
	\end{figure}

We compute the solution \eqref{3-s3} by three WG finite element methods with
 the same  $r_1$ and $r_2$ as in previous examples,
   on non-convex polygonal grids, shown in Figure  \ref{g2-1}.
The results are listed in Table \ref{t6-1}, 
   where we can see all solutions converge at the
  corresponding optimal orders. 
  Because we do not scale the solution and also because the jump of the
    coefficients is larger,  the errors here are much larger too.

  \begin{table}[H]
  \caption{ Error profile for computing \eqref{3-s3}.} \label{t6-1}
\begin{center}  
   \begin{tabular}{c|rr|rr}  
 \hline 
$G_i$ & 
       $\|\b u-\b u_h\|_{0}$&$O(h^r)$  &  $\|\nabla_w(\b u-\b u_h)\|_0$ & $O(h^r)$  \\ \hline  
     & \multicolumn{4}{c}{By the $P_1$ \eqref{Vk} WG element on Figure \ref{g2-1} grids}  \\
 4&    0.523E+01 &  3.1&    0.395E+03 &  1.7 \\
 5&    0.744E+00 &  2.8&    0.116E+03 &  1.8 \\
 6&    0.125E+00 &  2.6&    0.397E+02 &  1.5 \\
 \hline 
     & \multicolumn{4}{c}{By the $P_2$ \eqref{Vk} WG element on Figure \ref{g2-1} grids}  \\  
 4&    0.225E+01 &  4.1&    0.310E+03 &  2.9 \\
 5&    0.140E+00 &  4.0&    0.395E+02 &  3.0 \\
 6&    0.883E-02 &  4.0&    0.502E+01 &  3.0 \\
 \hline 
          & \multicolumn{4}{c}{By the $P_3$ \eqref{Vk} WG element on Figure \ref{g2-1} grids}  \\ 
 4&    0.567E+00 &  5.0&    0.127E+03 &  4.0 \\
 5&    0.176E-01 &  5.0&    0.797E+01 &  4.0 \\
 6&    0.548E-03 &  5.0&    0.499E+00 &  4.0 \\
 \hline 
\end{tabular} \end{center}  \end{table}

\long\def\skips#1{ }
\skips{
    \a{ 33 \b f_1|_{\Omega_0}
     & = -393216 x^{4} y^{2}+6291456 x^{3} y^{3}-1572864 x^{2} y^{4}+393216 x^{4} y\\
     &\quad \ -8650752 x^{3} y^{2}-6291456 x^{2} y^{3}+1572864 x \,y^{4}-77824 x^{4}\\
     &\quad \ +2949120 x^{3} y+11821056 x^{2} y^{2}+589824 x \,y^{3}-311296 y^{4}\\
     &\quad \ -139264 x^{3} -4841472 x^{2} y-3661824 x \,y^{2}+327680 y^{3}\\
     &\quad \ +349952 x^{2}+1849344 x y+72704 y^{2}  -160512 x-116736 y\\
     &\quad \ +13824,\\
     33 \b f_2|_{\Omega_0}
     & = 6291456 x^{4} y^{2}-1572864 x^{3} y^{3}+1572864 x^{2} y^{4}-6291456 x^{4} y\\
     &\quad \ -10223616 x^{3} y^{2}-786432 x^{2} y^{3}-1572864 x \,y^{4}+1245184 x^{4}\\
     &\quad \ +11649024 x^{3} y+5799936 x^{2} y^{2}+2211840 x \,y^{3}+311296 y^{4}\\
     &\quad \ -2416640 x^{3}-6365184 x^{2} y-1646592 x \,y^{2}-548864 y^{3}\\
     &\quad \ +1368064 x^{2}  +920064 x y+259072 y^{2}-189696 x-14592 y\\
     &\quad \ -3456,\\
      33 \b f_1|_{\Omega_1}
     & = 12288 x^{4}-589824 x^{3} y+368640 x^{2} y^{2}-589824 x \,y^{3}+49152 y^{4}\\
     &\quad \  +270336 x^{3}+516096 x^{2} y+516096 x \,y^{2}+196608 y^{3}\\
     &\quad \ -355584 x^{2}  -387072 x y-302592 y^{2}+155904 x+139776 y\\
     &\quad \ -25632,\\
      33 \b f_2|_{\Omega_1}
     & =   -196608 x^{4}+147456 x^{3} y-1474560 x^{2} y^{2}+147456 x \,y^{3}\\
     &\quad \ -49152 y^{4}+319488 x^{3}+1253376 x^{2} y+1253376 x \,y^{2}\\
     &\quad \ +24576 y^{3}-448512 x^{2}-1285632 x y-236544 y^{2}\\
     &\quad \ +304896 x +240384 y-52992 . }
  }


\end{document}